\documentclass[11pt]{amsart}
\pdfoutput=1

\usepackage[top=1in, bottom=1.25in, left=1.1in, right=1.1in]{geometry}
\usepackage[utf8]{inputenc}
\usepackage{subfiles}
\usepackage{hyperref}
\usepackage{amsthm}
\usepackage{amsmath}
\usepackage{amssymb}
\usepackage{amsfonts}
\usepackage{bm}
\usepackage{mathrsfs}
\usepackage{enumitem}
\usepackage{mathtools} 
\usepackage{tikz}
\usepackage{tikz-cd}
\usepackage{subcaption}
\usetikzlibrary{hobby} 
\usetikzlibrary{decorations.markings} 
\tikzset{->-/.style={decoration={markings,mark=at position #1 with {\arrow{>}}},postaction={decorate}}} 

\newtheorem{theorem}{Theorem}[section]
\theoremstyle{definition}
\newtheorem{proposition}[theorem]{Proposition}
\newtheorem{lemma}[theorem]{Lemma}
\newtheorem{definition}[theorem]{Definition}
\newtheorem{remark}[theorem]{Remark}
\newtheorem{corollary}[theorem]{Corollary}
\newtheorem{conjecture}[theorem]{Conjecture}

\newtheorem{example}[theorem]{Example}

\theoremstyle{remark}

\newcommand{\Z}{\mathbb Z}
\newcommand{\Q}{\mathbb Q}
\newcommand{\R}{\mathbb R} 
\newcommand{\C}{\mathbb C}

\renewcommand{\L}{\mathcal L}

\newcommand{\B}{\mathcal B}

\newcommand{\A}{\mathcal A}

\DeclareMathOperator{\Real}{Re}
\DeclareMathOperator{\Imag}{Im}

\DeclareMathOperator{\id}{id}

\DeclareMathOperator{\Li}{Li}

\DeclareMathOperator{\Ker}{Ker}

\DeclareMathOperator{\inv}{inv}
\DeclareMathOperator{\INV}{\boldsymbol{\iota}}

\DeclareMathOperator{\symb}{symb}

\usepackage[all]{xy}
\CompileMatrices
\def\cxymatrix#1{\xy*[c]\xybox{\xymatrix#1}\endxy}

\makeatletter
\newcommand*{\Relbarfill@}{\arrowfill@\Relbar\Relbar\Relbar}
\newcommand*{\xequal}[2][]{\ext@arrow 0055\Relbarfill@{#1}{#2}}
\makeatother

\title{The Lie coalgebra of multiple polylogarithms}

\author{Zachary Greenberg}
\address{Heidelberg University\\
         Mathematisches Instituts \\
         Heidelberg, 69120, Germany}
\email{zgreenberg@mathi.uni-heidelberg.de}

\author{Dani Kaufman}
\address{University of Copenhagen \\
         Department of Mathematical Sciences \\
         2100 Copenhagen \o, Denmark 
         \newline
         {\tt \url{https://sites.google.com/danikaufman/home}}}
\email{dk@math.ku.dk}

\author{Haoran Li}
\address{University of Maryland \\
         Department of Mathematics \\
         College Park, MD 20742-4015, USA}
\email{haoranli@umd.edu }

\author{Christian K. Zickert}
\address{University of Maryland \\
         Department of Mathematics \\
         College Park, MD 20742-4015, USA \newline
         {\tt \url{http://www.math.umd.edu/~zickert}}}
\email{zickert@umd.edu}

\thanks{C.~Z.~was supported in part by DMS-1711405 \\
\newline
2020 {\em Mathematics Classification.} Primary 11G55. 
Secondary 19E15, 14D07, 32G20. 
}

\begin{document}
\begin{abstract} We use Goncharov's coproduct of multiple polylogarithms to define a Lie coalgebra over an arbitrary field. It is generated by symbols subject to inductively defined relations, which we think of as functional relations for multiple polylogarithms. In particular, we have inversion relations and shuffle relations. We relate our definition to Goncharov's Bloch groups, and to the concrete model for $\L(F)_{\leq 4}$ by Goncharov and Rudenko.
\end{abstract}
\maketitle

\section{Introduction}
\subsection{The motivic Lie coalgebra}
For a field $F$, one expects the existence of a graded Lie coalgebra $\mathcal L(F)$ such that the weight $n$ part of its Chevalley-Eilenberg complex $\wedge^*(\mathcal L(F))$
computes the motivic cohomology groups $H^i_{\mathcal M}(F,\Z(n))$ (see e.g.~\cite{GoncharovMotivicGalois}). The existence of this so-called \emph{motivic Lie coalgebra} is known for number fields~\cite{Goncharov_GaloisSymmetriesOfFundamentalGroupoidsAndNoncommutativeGeometry}. It is desirable to have a concrete description. Goncharov~\cite{GoncharovMotivicGalois} conjectures that $\wedge^*(\mathcal L(F))_n$ is rationally quasi-isomorphic to the Bloch complex $\Gamma(F,n)$ defined in~\cite{GoncharovConfigurations}. This complex has the form
\begin{equation}\label{eq:GammaComplex}
\cxymatrix{@C=2em{\B_n(F)\ar[r]^-{\delta}&\cdots\ar[r]^-{\delta}&\B_{n-k}(F)\otimes\wedge^k(F^*)\ar[r]^-{\delta}&\cdots\ar[r]^-{\delta}&\B_2(F)\otimes\wedge^{n-2}(F^*)\ar[r]^-{\delta}&\wedge^n(F^*),}}
\end{equation}
with each group $\B_k(F)$ being generated by symbols $\{x\}_k$ with $x\in F\cup\{\infty\}$ subject to relations that may be thought of as polylogarithm relations. The rightmost $\delta$ takes $\{x\}_2\otimes a$ to $x\wedge(1-x)\wedge a$, and the others take $\{x\}_k\otimes a$ to $\{x\}_{k-1}\otimes x\wedge a$. Goncharov conjectures that $\mathcal L(F)_n=\B_n(F)$ for $n\leq 3$, but for $n=4$, $\mathcal L(F)_4$ is larger. An explicit model for $\wedge^*(\mathcal L(F))_{n\leq 4}$ is given in Goncharov-Rudenko~\cite{GoncharovRudenko} (see Section~\ref{sec:GoncharovRudenko} for a brief summary).

\subsection{Multiple polylogarithms and Goncharov's coproduct}
The classical polylogarithms $\Li_n(x)$ have multivariable generalizations $\Li_{n_1,\dots,n_d}(x_1,\dots,x_d)$ called \emph{multiple polylogarithms}. Their properties are extensively studied by Goncharov~\cite{Goncharov_MultiplePolylogarithmsAndMixedTateMotives,Goncharov_GaloisSymmetriesOfFundamentalGroupoidsAndNoncommutativeGeometry}. He shows that one may view $\Li_{n_1,\dots,n_d}(x_1,\dots,x_d)$ and $\log(x)$ as elements $\Li^{\mathscr C}_{n_1,\dots,n_d}(x_1,\dots,x_d)$ and $\log^{\mathscr C}(x)$ of a certain Hopf algebra (the Hopf algebra of framed rational Hodge structures~\cite[appendix]{Goncharov_GaloisSymmetriesOfFundamentalGroupoidsAndNoncommutativeGeometry}). In fact, they generate a Hopf subalgebra of this Hopf algebra, and Goncharov gives an explicit formula for the coproduct.
 It is best described using the generating series
\begin{equation}\label{eq:GenSeriesLi}
\begin{aligned}
    \Li^{\mathscr C}(x_1,\dots,x_d|t_1,\dots,t_d)=&\sum_{n_i\geq 1} \Li^{\mathscr C}_{n_1,\dots,n_d}(x_1,\dots,x_d)t_1^{n_1-1}\cdots t_d^{n_d-1},\\
    x^t=&\sum_{n\geq 0}\frac{1}{n!}\log^{\mathscr C}(x)^n t^n.
    \end{aligned}
\end{equation}
The coproduct of a generating series is defined termwise (fixing the $t_i$).
\begin{theorem}[{\cite[Prop.~6.1]{Goncharov_GaloisSymmetriesOfFundamentalGroupoidsAndNoncommutativeGeometry}}]\label{thm:GoncharovCoproduct} The coproduct $\Delta$ of $\Li^{\mathscr C}(x_1,\dots,x_d|t_1,\dots,t_d)$ is given by
\begin{equation}\label{eq:CoproductFormula}
\begin{aligned}
    \sum \Li^{\mathscr C}(x_{i_1\to i_2},\dots,x_{i_k\to i_{k+1}}|&t_{j_1},\dots t_{j_k})\bigotimes\\
    \prod_{\alpha=0}^k(-1)^{j_\alpha-i_\alpha}x_{i_\alpha\to i_{\alpha+1}}^{t_{j_\alpha}}\Big(&\Li^{\mathscr C}(x_{j_\alpha-1}^{-1},x_{j_\alpha-2}^{-1},\dots,x_{i_\alpha}^{-1}|t_{j_\alpha}-t_{j_\alpha-1},\dots, t_{j_\alpha}-t_{i_\alpha})\\
    &\cdot\Li^{\mathscr C}(x_{j_\alpha+1},x_{j_\alpha+2},\dots,x_{i_{\alpha+1}-1}|t_{j_\alpha+1}-t_{j_\alpha},\dots t_{i_{\alpha+1}-1}-t_{j_\alpha})\Big)
\end{aligned}
\end{equation}
The sum is over all $k=0,\dots, d$, and all sequences $\{i_\alpha\}_{\alpha=0}^{k+1}$ and $\{j_\alpha\}_{\alpha=0}^k$ with
\begin{equation}\label{eq:ijconditions}
    i_\alpha\leq j_\alpha<i_{\alpha+1},\qquad j_0=i_0=0,\quad i_{k+1}=d+1,
\end{equation}
and by definition we have $x_{i\to j}=\prod_{r=i}^{j-1}x_r$(note that this convension differs from Goncharov's, but works better in this paper) and $\Li^{\mathscr C}(\emptyset|\emptyset)=1$.
\end{theorem}


\begin{example}
When $d=2$, the $\{i_\alpha\}$ and $\{j_\alpha\}$ satisfying \eqref{eq:ijconditions} are:
\begin{equation}(\overset{i_0}{0},\overset{j_0}{0}|\overset{i_1}{1},\overset{j_1}{1}|\overset{i_2}{2},\overset{j_2}{2}|\overset{i_3}{3}),\quad  (\overset{i_0}{0},\overset{j_0}{0}|\overset{i_1}{1},\overset{j_1}{1}|\overset{i_2}{3}),\quad (\overset{i_0}{0},\overset{j_0}{0}|\overset{i_1}{1},\overset{j_1}{2}|\overset{i_2}{3}),\quad (\overset{i_0}{0},\overset{j_0}{0}|\overset{i_1}{2},\overset{j_1}{2}|\overset{i_2}{3}),\quad (\overset{i_0}{0},\overset{j_0}{0}|\overset{i_1}{3}),
\end{equation}
and it follows that
\begin{equation}\label{eq:CoprodDepth2}
\begin{aligned}
\Delta\Li^{\mathscr C}(x_1,x_2|t_1,t_2)=\Li^{\mathscr C}(x_1,x_2|t_1,t_2)\otimes x_1^{t_1}x_2^{t_2}+\Li^{\mathscr C}(x_1x_2|t_1)\otimes(x_1x_2)^{t_1}\Li^{\mathscr C}(x_2|t_2-t_1)\\
-\Li^{\mathscr C}(x_1x_2|t_2)\otimes(x_1x_2)^{t_2}\Li^{\mathscr C}(x_1^{-1}|t_2-t_1)+\Li^{\mathscr C}(x_2|t_2)\otimes\Li^{\mathscr C}(x_1|t_1)x_2^{t_2}+1\otimes\Li^{\mathscr C}(x_1,x_2|t_1,t_2).
\end{aligned}
\end{equation}
One can then compute $\Delta\Li^{\mathscr C}_{n_1,n_2}(x_1,x_2)$ as the coefficient of $t_1^{n_1-1}t_2^{n_2-1}$ of the right-hand side of~\eqref{eq:CoprodDepth2}. 
\end{example}

Given a graded Hopf algebra $H$, we have a Lie coalgebra $L=\frac{H_{>0}}{H_{>0}H_{>0}}$ with cobracket induced by the coproduct. We are thus interested in the quotient by products. The following elementary corollary of Theorem~\ref{thm:GoncharovCoproduct} is our main motivation.
\begin{corollary}\label{cor:CobracketLi} Modulo products and constants, $\Delta\Li^{\mathscr C}(x_1,\dots,x_d|t_1,\dots,t_d)$ can be written as 
\begin{equation}\label{eq:Cobracket}
\begin{aligned}
&\Li^{\mathscr C}(\mathbf x_{1,\dots,d}|\mathbf t_{1,\dots,d})\otimes \sum_{p=1}^dt_p\log^{\mathscr C}(x_p)+\\&\sum_{p=2}^d\Li^{\mathscr C}(\mathbf x_{p,\dots,d}|\mathbf t_{p,\dots,d})\otimes\Li^{\mathscr C}(\mathbf x_{1,\dots,p-1}|\mathbf t_{1,\dots,p-1})+\\
&\smashoperator{\sum_{1\leq p<q\leq d}} \Li^{\mathscr C}(\mathbf x_{1,\dots,p\to q,\dots,d}|\mathbf t_{1,\dots,p,q+1,\dots,d})\otimes \Li^{\mathscr C}(\mathbf x_{p+1,\dots, q}|\mathbf t_{p+1,\dots,q}-t_p)+\\
&\smashoperator{\sum_{1\leq p<q\leq d}}(-1)^{q-p}\Li^{\mathscr C}(\mathbf x_{1,\dots,p\to q,\dots,d}|\mathbf t_{1,\dots,p-1,q,\dots,d})\otimes\Li^{\mathscr C}(x_{q-1}^{-1},x_{q-2}^{-1},\dots, x_p^{-1}|t_q-\mathbf t_{q-1,\dots,p})
\end{aligned}
\end{equation}
Where $\mathbf x_{1,\dots,p\to q,\dots,d}$ is shorthand for $(x_1,\dots,x_{p-1},\prod_{r=p}^q x_r,x_{q+1},\dots,x_d)$ and dots indicate that indices increase (or decrease) by 1. We stress that the product is from $p$ to $q$, not to $q-1$.
\end{corollary}
\begin{proof}
One easily checks that the only sequences $\{i_\alpha\}$ and $\{j_\alpha\}$ for which the corresponding term in \eqref{eq:CoproductFormula} does not involve products are
\begin{equation}
\begin{aligned}
&(\overset{i_0}{0},\overset{j_0}{0}|\overset{i_1}{1},\overset{j_1}{1}|\overset{i_2}{2},\overset{j_2}{2}|\cdots|\overset{i_{d}}{d},\overset{j_{d}}{d}|,\overset{i_{d+1}}{d+1})\\
&(\overset{i_0}{0},\overset{j_0}{0}|\overset{i_1}{p},\overset{j_1}{p}|\overset{i_2}{p+1},\overset{j_2}{p+1}|\cdots|\overset{i_{k}}{d},\overset{j_{k}}{d}|,\overset{i_{k+1}}{d+1})\\
&(\overset{i_0}{0},\overset{j_0}{0}|\overset{i_1}{1},\overset{j_1}{1}|\cdots|\overset{i_p}{p},\overset{j_p}{p}|\overset{i_{p+1}}{q+1},\overset{j_{p+1}}{q+1}|\cdots|\overset{i_{k}}{d},\overset{j_{k}}{d}|,\overset{i_{k+1}}{d+1})\\
&(\overset{i_0}{0},\overset{j_0}{0}|\overset{i_1}{1},\overset{j_1}{1}|\cdots|\overset{i_p}{p},\overset{j_p}{q}|\overset{i_{p+1}}{q+1},\overset{j_{p+1}}{q+1}|\cdots|\overset{i_{k}}{d},\overset{j_{k}}{d}|,\overset{i_{k+1}}{d+1}),
\end{aligned}    
\end{equation}
and that the corresponding terms are exactly those given in~\eqref{eq:Cobracket}.
\end{proof}

\subsection{Real single valued polylogarithms}
One has single valued variants $\mathcal\L_n(x)$ of the classical polylogarithms (see e.g.~\cite{Deligne_InterpretationMotiviqueDeLaConjectureDeZagierReliantPolylogarithmesEtRegulateurs,ZagierDedekindZeta})
\begin{equation}
\L_n(z)=\begin{matrix}\Real \text{if $n$ is odd}\\\Imag \text{if $n$ is even}\end{matrix}\left (\sum_{r=0}^{n-1}\frac{2^rB_r}{r!}\Li_{n-r}(z)(\log\vert z\vert)^r\right ).
\end{equation}
The following result justifies thinking about the relations in $\B_n(F)$ as polylogarithm relations:
\begin{theorem}[{\cite[Prop.~3]{ZagierDedekindZeta}}]\label{thm:RnAndLn} For any element $\alpha=\sum n_i{f_i(t)}\in \B_n(\C(t))$ with $\delta(\alpha)=0$ in $\B_{n-1}(\C(t))\otimes\C(t)^*$ we have
\begin{equation}
    \sum n_i\L_n(f_i(t))=\textnormal{constant}.
\end{equation}
\end{theorem}
There are also single valued analogues $\L_{n_1,\dots,n_d}$ of the multiple polylogarithms $\Li_{n_1,\dots,n_d}$. There is no closed formula, but they can be computed from a variation matrix~\cite{ZhaoMultiplePolylogsMonodromyHodgeStructures}. We shall here only need that $\L_{n_1,\dots,n_d}$ is defined on 
\begin{equation}\label{eq:SdC}
S_d(\C)=\Big\{(x_1,\dots,x_d)\in(\C^*)^d\bigm\vert  \prod_{r=i}^j x_r\neq 1 \text{ for all }i\leq j\in\{1,\dots,d\}\Big\}
\end{equation}
and that
\begin{equation}\label{eq:ZeroLimit}
    \lim_{x_i\to 0}\mathcal L_{n_1,\dots,n_d}(x_1,\dots,x_d)=0\qquad \text{for all $i$}.
\end{equation}
The functions $\L_n(x)$ also satisfy that $\lim_{x\to\infty}\L_n(x)=0$ for $n>1$, but in higher depth the limit as $x_i$ tends to $\infty$ is no longer $0$.


\subsection{Structure of the paper}
In Section~\ref{sec:Lsymb} we define a purely symbolic coalgebra $\mathbb L^{\symb}(F)$ with no relations. We have $\mathbb L_1^{\symb}(F)=F^*$, and $\mathbb L^{\symb}_{>1}$ is generated by symbols $[x_1,\dots,x_d]_{n_1,\dots,n_d}$ with $(x_1,\dots,x_d)\in S_d(F)$, where $d$ and $n_1,\dots, n_d$ are positive integers. Thinking of a symbol as a polylogarithm we define the cobracket as in Corollary~\ref{cor:CobracketLi} and show directly that $\delta^2=0$. In Section~\ref{sec:Relations} we inductively define a group of relations $R_n(F)$ and define 
\begin{equation}
\mathbb L_n(F)=\mathbb L_n^{\symb}(F)/R_n(F).
\end{equation}
The definition mimics Goncharov's definition of relations in $\B_n(F)$, and the proof that the cobracket takes relations to 0 in $\wedge^2(\mathbb L(F))$ follows Goncharov as well. Section~\ref{sec:BasicExamples} gives some basic examples of relations, and Section~\ref{sec:NotAllSymbols} discusses the problem of defining symbols when $(x_1,\dots,x_d)\notin S_d(F)$. For example, $[1,1]_{1,1}$ is not well defined.
In Section~\ref{sec:Speculation} we speculate that $\mathbb L(F)$ is the motivic Lie coalgebra $\L(F)$ and conjecture a generalization of Theorem~\ref{thm:RnAndLn}, which justifies thinking of $R_n(F)$ as polylogarithm relations.
Section~\ref{sec:InversionRelations} discusses the \emph{inversion relations}, which are inspired by Goncharov's inversion relations for multiple polylogarithms. In particular, we show that in $\mathbb L(F)$ one can express each symbol $[x_d^{-1},\dots,x_1^{-1}]_{n_d,\dots,n_1}$ in terms of symbols involving only non-inverted $x_i$. Section~\ref{sec:AltCobracket} shows that one can use this to define an alternative cobracket on $\mathbb L^{\symb}(F)$ without inverted $x_i$. The relations are the same. The alternative one may in fact be more natural (see e.g.~Remark~\ref{rm:ShuffleZeroInLsymb}).
Section~\ref{sec:ShuffleRelations} briefly discusses the shuffle product relations showing that at least in low depth a shuffle product is 0 in $\mathbb L(F)$. Finally, Section~\ref{sec:GoncharovRudenko} relates our work to that of Goncharov and Rudenko, who gave a concrete model for $\L_{n\leq 4}(F)$.

\begin{remark} Although our work is heavily inspired by the work of Goncharov, it does not require any of Goncharov's results, except for motivation.
\end{remark}

\subsection*{Acknowledgment} We thank Lars Hesselholt for helpful comments. C.~Z.~was funded in part by NSF grant DMS-1711405.

\section{A purely symbolic Lie coalgebra}\label{sec:Lsymb}
For a field $F$, and a positive integer $d$, let
\begin{equation}\label{eq:SdDef}
S_d(F)=\Big\{(x_1,\dots,x_d)\in(F^*)^d\bigm\vert  \prod_{r=i}^j x_r\neq 1 \text{ for all }i\leq j\in\{1,\dots,d\}\Big\}.
\end{equation}
We wish to define a graded Lie coalgebra
\begin{equation}
    \mathbb L^{\symb}(F)=\bigoplus_{n=1}^{\infty} \mathbb L_n^{\symb}(F).
\end{equation}
We first define $\mathbb L_1^{\symb}(F)=F^*$,
which we shall identify with the abelian group generated by symbols $[x]_1$ with $x\in F\setminus \{0,1\}$ and $[x]_0$ for $x\in F\setminus \{0,1\}$ subject to the relations
\begin{equation}
   [x]_1=-[1-x]_0,\qquad [xy]_0=[x]_0+[y]_0. 
\end{equation}
For $n>1$, define $\mathbb L_n^{\symb}(F)$ to be generated by symbols $[x_1,\dots,x_d]_{n_1,\dots,n_d}$ with 
\begin{equation}
    (x_1,\dots,x_d)\in S_d(F),\qquad d,n_1,\dots,n_d\in\Z_+,\qquad n_1+\dots+n_d=n.
\end{equation}
We refer to $n_1+\dots+n_d$ as the \emph{weight} and $d$ as the \emph{depth} of a symbol, but stress that $[x]_0$ is in weight 1, not 0. We think of $[x_1,\dots,x_d]_{n_1,\dots,n_d}$ as representing $\Li_{n_1,\dots,n_d}(x_1,\dots,x_d)$ and $[x]_0$ as representing $\log(x)$.




\subsection{The cobracket on $\mathbb L^{\symb}(F)$}
As in~\eqref{eq:GenSeriesLi} we define
\begin{equation}
[x_1,\dots,x_d|t_1,\dots,t_d]=\sum_{n_i\geq 1} [x_1,\dots,x_d]_{n_1,\dots,n_d}t_1^{n_1-1}\cdots t_d^{n_d-1},
\end{equation}
and define $\delta\colon\mathbb L^{\symb}(F)\to\wedge^2(\mathbb L^{\symb}(F))$ to be zero in weight 1, and otherwise given by~\eqref{eq:Cobracket}. It is convenient to write $\delta=\delta_1+\delta_2+\delta_3+\delta_4$, where for $X=[\mathbf x_{1,\dots,d}|\mathbf t_{1,\dots,d}]$
\begin{equation}
\begin{aligned}
\delta_1(X)&=X\wedge \left(\sum_{p=1}^dt_p[x_p]_0\right)\\
\delta_2(X)&=\sum_{p=2}^d[\mathbf x_{p,\dots,d}|\mathbf t_{p,\dots,d}]\wedge[\mathbf x_{1,\dots,p-1}|\mathbf t_{1,\dots,p-1}]\\
\delta_3(X)&=\sum_{1\leq p<q\leq d} [\mathbf x_{1,\dots,p\to q,\dots,d}|\mathbf t_{1,\dots,p,q+1,\dots,d}]\wedge [\mathbf x_{p+1,\dots, q}|\mathbf t_{p+1,\dots,q}-t_p]\\
\delta_4(X)&=\sum_{1\leq p<q\leq d}(-1)^{q-p}[\mathbf x_{1,\dots,p\to q,\dots,d}|\mathbf t_{1,\dots,p-1,q,\dots, d}]\wedge[x_{q-1}^{-1},\dots,x_p^{-1}|t_q-\mathbf t_{q-1,\dots,p}].
\end{aligned}
\end{equation}

\begin{example} In depth 1, we have $\delta[x|t]=[x|t]\wedge t[x]_0$, from which it follows that $\delta[x]_n=[x]_{n-1}\wedge [x]_0$ for $n>1$. Note the case $\delta[x]_2=[x]_1\wedge[x]_0=x\wedge (1-x)\in\wedge^2(F^*)$. We thus recover the boundary map in the Bloch complex.
\end{example}
\begin{example} In depth 2, $\delta[x_1,x_2|t_1,t_2]$ equals (compare with~\eqref{eq:CoprodDepth2})
\begin{equation}
\begin{aligned}
\delta[x_1,x_2|t_1,t_2]=[x_1,x_2|t_1,t_2]\wedge(t_1[x_1]_0+t_2[x_2]_0)+[x_2|t_2]\wedge [x_1|t_1]+\\{[x_1x_2|t_1]\wedge[x_2|t_2-t_1]}-[x_1x_2|t_2]\wedge[x_1^{-1}|t_2-t_1].
\end{aligned}
\end{equation}
In particular, $\delta[x_1,x_2]_{r,s}$ is given by
\begin{equation}
\begin{aligned}
&[x_1,x_2]_{r-1,s}\wedge[x_1]_0+[x_1,x_2]_{r,s-1}\wedge[x_2]_0+[x_2]_s\wedge[x_1]_r\\
&+\sum_{i=1}^r(-1)^{r-i}\binom{r+s-1-i}{s-1}[x_1x_2]_i\wedge[x_2]_{r+s-i}\\
&+\sum_{i=1}^s (-1)^{r}\binom{r+s-1-i}{r-1}[x_1x_2]_i\wedge[x_1^{-1}]_{r+s-i}.
\end{aligned}
\end{equation}
In the case where $r$ or $s$ is 1, the symbols $[x_1,x_2]_{0,s}$ and $[x_1,x_2]_{r,0}$ are interpreted as 0.
\end{example}

\begin{remark} In our definition of $\wedge^2(\mathbb L^{\symb}(F))$, $x\wedge x$ is identically 0, not 2-torsion.
\end{remark}


\begin{theorem}\label{thm:cobracket} The cobracket above makes $\mathbb L^{\symb}(F)$ into a Lie coalgebra.
\end{theorem}
\begin{proof}
We must show that $\delta^2=0$. To do this, it is enough to show that
\begin{equation}
\begin{gathered}
\delta_1^2=0,\qquad\delta_2^2=0,\qquad\delta_1\delta_i+\delta_i\delta_1=0,\\
\delta_3^2+\delta_2\delta_3+\delta_3\delta_2=0,\qquad \delta_2\delta_4+\delta_4\delta_2+\delta_4^2+\delta_3\delta_4+\delta_4\delta_3=0.
\end{gathered}
\end{equation}
 The proof that $\delta_1^2=0$ is elementary. The remaining equalities are all straightforward, so for brevity, we prove only that $\delta_2^2=0$, and that $\delta_1\delta_4+\delta_4\delta_1=0$. Firstly, $\delta_2^2([\mathbf x_{1,\dots,d}|\mathbf t_{1,\dots,d}])$ is given by
\begin{equation}
\begin{aligned}
&\sum_{2\leq p\leq d}\delta_2[\mathbf x_{p,\dots,d}|\mathbf t_{p,\dots ,d}]\wedge[\mathbf x_{1,\dots,p-1}|\mathbf t_{1,\dots ,p-1}]-\sum_{2\leq p\leq d}[\mathbf x_{p,\dots,d}|\mathbf t_{p,\dots ,d}]\wedge\delta_2[\mathbf x_{1,\dots,p-1}|\mathbf t_{1,\dots,p-1}] \\
&=\sum_{2\leq p\leq d}\sum_{p+1\leq r\leq d}[\mathbf x_{r,\dots,d}|\mathbf t_{r,\dots ,d}]\wedge[\mathbf x_{p,\dots,r-1}|\mathbf t_{p,\dots ,r-1}]\wedge[\mathbf x_{1,\dots,p-1}|\mathbf t_{1,\dots,p-1}]\\
&-\sum_{2\leq p\leq d}\sum_{2\leq r\leq p-1}[\mathbf x_{p,\dots,d}|\mathbf t_{p,\dots ,d}]\wedge[\mathbf x_{r,\dots,p-1}|\mathbf t_{r,\dots ,p-1}]\wedge[\mathbf x_{1,\dots,r-1}|\mathbf t_{1,\dots ,r-1}]=0.
\end{aligned}
\end{equation}
Similarly, $\delta_1\delta_4([\mathbf x_{1,\dots,d}|\mathbf t_{1,\dots,d}])$ equals
\begin{equation}
\begin{aligned}
&\sum_{1\leq p<q\leq d}(-1)^{q-p}\delta_1[\mathbf x_{1,\dots,p\to q,\dots,d}|\mathbf t_{1,\dots ,p-1,q,\dots ,d}]\wedge[x_{q-1}^{-1},\dots,x_p^{-1}|t_q-\mathbf t_{q-1,\dots,p}]\\
&-(-1)^{q-p}[\mathbf x_{1,\dots ,p\to q,\dots ,d}|\mathbf t_{1,\dots ,p-1,q,\dots,d}]\wedge\delta_1[x_{q-1}^{-1},\dots,x_p^{-1}|t_q-\mathbf t_{q-1,\dots,p}],
\end{aligned}
\end{equation}
which equals
\begin{equation}
    \begin{aligned}
&=\sum _{1\leq p<q\leq d} (-1)^{q-p}[\mathbf x_{1,\dots,p\to q,\dots,d}|\mathbf t_{1,\dots ,p-1,q,\dots ,d}]\wedge\\
&\left(\sum_{1\leq r\leq p-1} t_r[x_r]_0+t_q[x_p\cdots x_q]_0+\sum_{q+1\leq r\leq d} t_r[x_r]_0\right)\wedge[x_{q-1}^{-1},\dots,x_p^{-1}|t_q-\mathbf t_{q-1,\dots,p}]\\
&-(-1)^{q-p}[\mathbf x_{1,\dots ,p\to q,\dots ,d}|\mathbf t_{1,\dots ,p-1,q,\dots,d}]\wedge[x_{q-1}^{-1},\dots,x_p^{-1}|t_q-\mathbf t_{q-1,\dots,p}]\wedge\left(\sum_{p\leq r\leq q-1}(t_q-t_r)[x_r^{-1}]_0\right)\\
&=-\sum _{1\leq p<q\leq d} (-1)^{q-p}[\mathbf x_{1,\dots,p\to q,\dots,d}|\mathbf t_{1,\dots,p-1,q,\dots,d}]\wedge[x_{q-1}^{-1},\dots,x_p^{-1}|t_q-\mathbf t_{q-1,\dots,p}]\wedge\left(\sum_{1\leq r\leq d} t_r[x_r]_0\right).
\end{aligned}
\end{equation}
Finally,
\begin{multline}
\delta_4\delta_1([\mathbf x_{1,\dots,d}|\mathbf t_{1,\dots,d}])=\delta_4[\mathbf x_{1,\dots,d}|\mathbf t_{1,\dots,d}]\wedge \left(\sum_{p=1}^dt_p[x_p]_0\right)=\\
=\sum_{1\leq p<q\leq d}(-1)^{q-p}[\mathbf x_{1,\dots,p\to q,\dots d}|\mathbf t_{1,\dots,p-1,q,\dots d}]\wedge[x_{q-1}^{-1},\dots,x_p^{-1}|t_q-\mathbf t_{q-1,\dots,p}]\wedge\left(\sum_{p=1}^dt_p[x_p]_0\right),\end{multline}
and it follows that $\delta_1\delta_4+\delta_4\delta_1=0$.
\end{proof}

\begin{remark} Theorem~\ref{thm:cobracket} also holds without requiring that the tuples are in $S_d(F)$ (same proof). The problem with arbitrary tuples arises when defining the relations; see Section~\ref{sec:NotAllSymbols}.
\end{remark}

\subsection{Terms involving zero or infinity}\label{sec:ZeroInfinity} We shall also define elements $[x_1,\dots,x_d]_{n_1,\dots,n_d}$ when some of the $x_i$ are $0$ or $\infty$, but we still require that all consecutive products $\prod_{r=i}^j x_r$ are well defined and not 1 ($\infty x_i=\infty\neq 1$; $0\infty$ is undefined). When some $x_i$ are 0, $[x_1,\dots,x_d]_{n_1,\dots,n_d}$ is defined to be zero (this is motivated by~\eqref{eq:ZeroLimit}). When some $x_i$ are $\infty$ the definition is more subtle and we refer to Section~\ref{sec:InversionRelations}. All we need for now is that $[\infty]_n=0$ for $n>1$ (and undefined for $n$=1).

\section{The relations}\label{sec:Relations}
We now define groups $R_n(F)$ of relations in $\mathbb L_n^{\symb}(F)$. We can then define 
\begin{equation}\mathbb L_n(F)=\mathbb L_n^{\symb}(F)/R_n(F).
\end{equation}
The definition is inductive starting with the definition of $R_1(F)$ to be the trivial group, so that $\mathbb L_1(F)=F^*$. Suppose $n>1$ and that $\mathbb L_k(K)$ has been defined for all fields $K$ and all $k<n$. Then 
\begin{equation}\wedge^2(\mathbb L(K))_n=\bigoplus_{k+l=n}\mathbb L_k(K)\wedge \mathbb L_l(K)
\end{equation}
is also well defined for all $K$. Let
\begin{equation}
\mathcal A_n(K)=\Ker\Big(\delta\colon\mathbb L_n^{\symb}(K)\to\wedge^2(\mathbb L(K))_n\Big)
\end{equation}
Mimicking Goncharov's definition of relations in the Bloch complexes~\cite[p.~221]{GoncharovConfigurations} we wish to define $R_n(F)\subset \mathbb L_n^{\symb}(F)$ to be generated by elements $\alpha(p)-\alpha(q)$, where $p$ and $q$ are points on a connected (geometrically irreducible) smooth curve $X$ over $F$ with function field $F(X)$ and $\alpha$ is an element in $\mathcal A_n(F(X))$. The only problem with this is our requirement that all tuples be well defined (allowing for $0$ and $\infty$; see Section~\ref{sec:ZeroInfinity}).

Each element $\alpha\in \mathbb L^{\symb}(F(X))$ is a linear combination of terms $[f_1,\dots,f_d]_{n_1,\dots,n_d}$, which we refer to as \emph{terms of $\alpha$}.
\begin{definition} We say that $\alpha\in\mathbb L^{\symb}(F(X))$ is \emph{well defined} at $x\in X$ if all products $\prod_{r=i}^jf_r(x)$ are defined and distinct from 1, for each term $[f_1,\dots,f_d]_{n_1,\dots,n_d}$ of $\alpha$.
\end{definition}
By induction, we see that if $\alpha\in\mathcal A_n(F(X))$, we can rearrange the terms $\delta(\alpha)$ as a linear combination of terms $(\beta(q)-\beta(q'))\otimes y$, where $Y$ is a smooth connected curve over $F(X)$, $q$ and $q'$ are points in $Y$, $\beta\in\mathcal A_k(F(X)(Y))$ and $y\in\mathbb L^{\symb}_{n-k}(F(X))$.

\begin{definition}
All $\alpha \in \mathcal A_2$ are \emph{proper}. For $n>2$, $\alpha\in \mathcal A_n(F(X))$ is \emph{proper} if whenever $\alpha$ is well defined at $x\in X$, there is an arrangement of $\delta(\alpha)$ as above with $\beta(q)$, $\beta(q')$, and $y$ all well defined at $x$ (no cancelation of undefined terms) with $\beta$ proper.
\end{definition}

\begin{definition} The group $R_n(F)$ is generated by elements of the form $\alpha(p)-\alpha(q)$, where $\alpha\in\mathcal A_n(F(X))$ is proper and well defined at $p,q\in X$.
\end{definition}



We now prove that $\mathbb L(F)$ is also a Lie coalgebra. To do this we must prove that the cobracket respects the relations. The proof is similar to~\cite[Lemma~1.16]{GoncharovConfigurations}.



\begin{theorem} The map $\delta\colon\mathbb L_n^{\symb}(F)\to \wedge^2(\mathbb L(F))_n$ takes $R_n(F)$ to 0.
\end{theorem}
\begin{proof}
It is enough to show that for any proper $\alpha\in\mathcal A_n(F(X))$, the element $\delta(\alpha(p))$ is zero in $\wedge^2(\mathbb L(F))_n$ for all $p\in X$ where $\alpha(p)$ is defined. Fix such $X$, $\alpha$, and $p$ and write $\delta(\alpha)$ in $\wedge^2(\mathbb L^{\symb}(F(X)))$ as a linear combination of terms $(\beta(q)-\beta(q'))\wedge y$ as above. It follows that $\delta(\alpha(p))$ is a linear combination of elements of the form $(\beta(q)(p)-\beta(q')(p))\otimes y(p)$, and the result follows by showing that $\beta(q)(p)-\beta(q')(p)$ is in $R_k(F)$. Let $\beta_p$ be the element in $\mathbb L_k^{\symb}(F(Y(p)))$ obtained from $\beta$ by restriction to the fiber $Y(p)$ over $p$. Since $\beta$ is in $\mathcal A_k(F(X)(Y)))$, an induction argument shows that $\beta_p$ is in $\mathcal A_k(F(Y(p))$ and is proper. Note that $q$ and $q'$ can be regarded as maps $X\to Y$, so restriction to $p\in X$ determines points $q_p$ and $q'_p$ in $Y(p)$. We then have
\begin{equation}
   \beta(q)(p)-\beta(q')(p)=\beta_p(q_p)-\beta_p(q_p)\in R_k(F).
\end{equation}
This concludes the proof.
\end{proof}

\subsection{Basic examples of relations}\label{sec:BasicExamples}

In the following examples we repeatedly use the fact that for $\alpha \in \mathbb L^{\symb}(F(x))$, $\delta(\alpha) = 0 \in\wedge^2(\mathbb L(F(x)))$ implies that $\alpha(x)$ is constant in $\mathbb L(F)$.

\begin{example}\label{ex:Depth1Inv}
Let $\alpha_n=2([x]_n+(-1)^n[x^{-1}]_n)\in\mathbb L^{\symb}_n(F)$. By induction, $\alpha_n\in\A_n(F(x))$ when $n\geq 2$, so $\alpha_n(x)$ is constant in $\mathbb L_n(F)$. Hence, $\alpha_n(x)=\alpha_n(0)=2[0]_n+2(-1)^n[\infty]_n=0$. Note, however, that $\alpha_1(x)=-2[x]_0$.
\end{example}

\begin{example}\label{ex:11} One easily shows that $\alpha=[x,y]_{1,1}+[x]_2-[\frac{x(1-y)}{1-xy}]_2$ is in $\A_2(F(x,y))$. Considering the specialization at $(x,y)=(0,y)$ it follows that $[x,y]_{1,1}+[x]_2-[\frac{x(1-y)}{1-xy}]_2=0\in\mathbb L_2(F)$ whenever all terms are defined. This means that every element in $\mathbb L_2(F)$ can be expressed using only terms in depth 1.
\end{example}

\begin{example}[Five term relation] Similarly, we obtain that \begin{equation}[x]_2+[y]_2-[xy]_2-[\frac{y(1-x)}{1-xy}]_2-[\frac{x(1-y)}{1-xy}]_2=0\in\mathbb L_2(F).
\end{equation}
\end{example}

\begin{proposition}
Modulo 2-torsion every element in $\mathbb L_3(F)$ can be written in terms of depth 1 symbols.
\end{proposition} 
\begin{proof} A simple computation shows that the elements
\begin{multline}\label{eq:21InTermsOf3}
    [x,y]_{2,1}+[x]_3+[\frac{1-y}{1-xy}]_3+[xy]_3+[-\frac{xy}{1-xy}]_3-[1-y]_3-[\frac{x(1-y)}{1-xy}]_3
\end{multline}
and
\begin{multline}\label{eq:111InTermsOf3}
    [x,y,z]_{1,1,1}-[-\frac{y}{1-y}]_3+[\frac{1-x}{1-xyz}]_3-[xy]_3+[\frac{xy(1-z)}{1-xyz}]_3\\
-[1-x]_3+[-\frac{y(1-x)}{1-y}]_3+[-\frac{y(1-z)}{1-y}]_3-[-\frac{y(1-x)(1-z)}{(1-xyz)(1-y)}]_3
\end{multline}
have vanishing cobracket modulo 2-torsion. Specializing at $(0,y)$ and $(x,0,z)$, respectively, proves the result for $[x,y]_{2,1}$ and $[x,y,z]_{1,1,1}$. Finally, $[x,y]_{1,2}+[y,x]_{2,1}+[xy]_3$ also has vanishing cobracket, concluding the proof.
\end{proof}

\begin{remark} With more work one can show that (up to torsion) every element in $\mathbb L_4(F)$ can be expressed in terms of terms of the form $[x]_4$ and $[x,y]_{3,1}$. For example, one has 
\begin{equation}
    [x,y]_{2,2}=[y]_4+[xy]_4+[y,x]_{3,1}+[xy,x^{-1}]_{3,1}-[x,y]_{3,1}.
\end{equation}
Similar equations for $[x,y,z]_{2,1,1}$, $[x,y,z]_{1,2,1}$, $[x,y,z]_{1,1,2}$, and $[x,y,z,w]_{1,1,1,1}$ are more complicated, and we omit them.
\end{remark}

\subsection{The issue of arbitrary symbols}\label{sec:NotAllSymbols}
The following argument shows that not all symbols can be meaningfully defined. By Example~\ref{ex:11} we should have for any $a$
\begin{equation}
    [x,\frac{x-a}{x-ax}]_{1,1}+[x]_2-[a]_2=0.
\end{equation}
Setting $x=1$, we would thus get that $[1,1]_{1,1}+[1]_2=[a]_2$ for any $a$, which would imply that $\mathbb L_2(F)=0$.
\begin{remark} We believe that the reason for this is that $\L_{1,1}(x,y)$ does not have a limit as $x$ and $y$ tend to 1. One has similar issues with symbols such as $[x_1,0,\infty]_{1,1,1}$.
\end{remark}
\begin{remark} One can give meaning to some additional symbols, but these should always be expressible in terms of symbols in $S_d(F)$. For example, one may define $[x,x^{-1}]_{1,1}=-[x]_2$. We shall not pursue this here.
\end{remark}


\subsection{Conjectures and speculation}\label{sec:Speculation}

%



Conjecture~\ref{conj:RnAndLnGeneral} below is a natural generalization of Theorem~\ref{thm:RnAndLn}. Consider the map
\begin{equation}
    r\colon\mathbb L^{\symb}(\C)\to\R,\qquad [x_1,\dots,x_d]_{n_1,\dots,n_d}\to\L_{n_1,\dots,n_d}(x_1,\dots,x_d),\qquad [x]_0\mapsto \log(|x|).
\end{equation}
\begin{conjecture}\label{conj:RnAndLnGeneral}
If $\alpha\in\mathbb L(\C(t))$ is such that $\delta(\alpha)=0$, then $r(\alpha(t))=0$ is constant in $t$. 
\end{conjecture}

\begin{remark} We have verified~Conjecture~\ref{conj:RnAndLnGeneral} numerically for many examples in weight 4 and lower.
\end{remark}

We also speculate more boldly that $\mathbb L(F)$ is rationally isomorphic to the motivic Lie coalgebra, so that the weight $n$ part of the complex
\begin{equation}
     \cxymatrix{@C=4em{\mathbb L(F)\ar[r]^-\delta&\wedge^2(\mathbb L(F))\ar[r]^-{\delta\wedge \id - \id\wedge\delta}&\wedge^3(\mathbb L(F))\ar[r]&\dots}}
\end{equation}
should compute motivic cohomology, rationally. For a finite field $F_q$ with $q$ elements the only motivic cohomology groups are
\begin{equation}
    H^1_{\mathcal M}(F_q,\Z(n))=K_{2n-1}(F_q)=\Z\big/(q^n-1)\Z.
\end{equation}
Assuming that $R_2(F_q)$ is generated by five term relations and that $R_3(F_q)$ is generated by the 31 term relations in~\cite{Zickert_HolomorphicPolylogarithmsAndBlochComplexes}, one can compute $H^1(\bigwedge^*(\mathbb L(F_q))_n)$ for $n=2$ and $3$. When $q>3$, experimental evidence suggests that \emph{integrally}, one has 
\begin{equation}
    |H^1(\textstyle\bigwedge\nolimits^*\displaystyle\big(\mathbb L(F_q)\big)_n)|=_{2,3}\frac{|H^1_{\mathcal M}(F_q,\Z(n))|}{|F_q^*|}
\end{equation}
where $=_{2,3}$ means equality up to small powers of 2 and 3. The equality appears to hold on the nose when $n=2$.

\section{Inversion relations}\label{sec:InversionRelations}
In the following we ignore 2-torsion. It is well known (see e.g.~\cite{PolylogBook}) that the classical polylogarithms satisfy the \emph{inversion relation}
\begin{equation}\label{eq:Depth1Inv}
\Li_n(x)+(-1)^{n}\Li_n(x^{-1})=-\frac{(2\pi i)^n}{n!}B_n\left(\frac{\log x}{2\pi i}\right),
\end{equation}
where $B_n(x)=\sum_{k=0}^n\binom{n}{k}B_{n-k}x^k$ are the Bernoulli polynomials. Thinking of a symbol $[x]_n$ as $\Li_n(x)$ modulo products and powers of $\pi i$, we would thus expect $[x]_n+(-1)^n[x^{-1}]_n$ to be zero in $\mathbb L_n(F)$. This was proved in Example~\ref{ex:Depth1Inv}.

Goncharov~\cite[Sec.~2.6]{Goncharov_MultiplePolylogarithmsAndMixedTateMotives} extended the classical inversion formula to multiple polylogarithms. For example, one has
\begin{equation}
\begin{aligned}
\Li_{n_1,n_2}(x_1,x_2)+(-1)^{n_1+n_2}\Li_{n_2,n_1}(x_2^{-1},x_1^{-1})+(-1)^{n_1}\Li_{n_1}(x_1^{-1})\Li_{n_2}(x_2)\\
+\sum_{p+q=n_1}\frac{(2\pi i)^p}{p!}(-1)^q\binom{q+n_2-1}{n_2-1}B_p\left(\frac{\log(x_1x_2)}{2\pi i}\right)\Li_{q+n_2}(x_2)\\
+\sum_{p+q=n_2}\frac{(2\pi i)^p}{p!}(-1)^{n_1}\binom{n_1+q-1}{n_1-1}B_p\left(\frac{\log(x_1x_2)}{2\pi i}\right)\Li_{n_1+q}(x_1^{-1})=0,
\end{aligned}
\end{equation}
which suggests that we should have
\begin{multline}\label{eq:InvDept2}
[x_1,x_2]_{n_1,n_2}+(-1)^{n_1+n_2}[x_2^{-1},x_1^{-1}]_{n_2,n_1}\\
+(-1)^{n_1}\binom{n_1+n_2-1}{n_2-1}[x_2]_{n_1+n_2}
+(-1)^{n_1}\binom{n_1+n_2-1}{n_1-1}[x_1^{-1}]_{n_1+n_2}=0\in \mathbb L_{n_1+n_2}(F).
\end{multline}
More generally, Goncharov's work suggests that $[x_d^{-1},\dots,x_1^{-1}]_{n_d,\dots,n_1}$ should be expressible using terms where no $x_i$ is inverted. Inspired by Goncharov's work we make the following definition.

\begin{definition} The \emph{inversion map} is defined inductively on power series by the formula
\begin{multline}\label{eq:InvRel}
\inv([\mathbf x_{1,\dots,d}|\mathbf t_{1,\dots,d}])=-(-1)^d[\mathbf x_{1,\dots,d}|\mathbf t_{1,\dots,d}]+\frac{(-1)^d}{t_1}[\mathbf x_{2,\dots,d}|\mathbf t_{2,\dots,d}]\\-\frac{(-1)^d}{t_1}[\mathbf x_{2,\dots,d}|\mathbf t_{2,\dots,d}-t_1]-\frac{1}{t_d}\inv([\mathbf x_{1,\dots,d-1}|\mathbf t_{1,\dots,d-1}])+\frac{1}{t_d}\inv([\mathbf x_{1,\dots,d-1}|\mathbf t_{1,\dots,d-1}-t_d]).
\end{multline}
The induction starts with the formula $\inv[x|t]=[x|t]+[x]_0$.
\end{definition}

We wish to prove the following:
\begin{theorem}\label{thm:InvRel} We have $[\mathbf x^{-1}_{d,\dots,1}|-\mathbf t_{d,\dots,_1}]=\inv[\mathbf x_{1,\dots,d}|\mathbf t_{1,\dots,d}]\in\mathbb L(F)$. In particular,
\begin{equation}
(-1)^d(-1)^{n_1+\dots+n_d}[x_d^{-1},\dots,x_1^{-1}]_{n_d,\dots,n_1}=\inv[x_1,\dots,x_d]_{n_1,\dots,n_d}.
\end{equation}
\end{theorem}

\begin{example}
For $d=1$, we have $(-1)^{n+1}[x^{-1}]_n=[x]_n$ for $n>1$ and $[x^{-1}]_1=[x]_1+[x]_0$. These hold by Example~\ref{ex:Depth1Inv}.
\end{example}

\begin{example}
For $d=2$, $\inv[x_1,x_2|t_1,t_2]$ equals
\begin{equation}
-[x_1,x_2|t_1,t_2]+\frac{1}{t_1}[x_2|t_2]-\frac{1}{t_1}[x_2|t_2-t_1]-\frac{1}{t_2}\inv[x_1|t_1]+\frac{1}{t_2}\inv[x_1|t_1-t_2].
\end{equation}
The coefficient $t_1^{n_1-1}t_2^{n_2-1}$ of the equation in Theorem~\ref{thm:InvRel}, gives \eqref{eq:InvDept2}.
\end{example}

Recall that the proof that $[x]_n+(-1)^n[x^{-1}]_n=0$ used that the symbol $[\infty]_n$ is well defined and equal to 0. Unfortunately, $\L_{n_1,\dots,n_d}(x_1,\dots,x_d)$ does not tend to 0 when some of the $x_i$ tend to $\infty$, so we do not expect to be able to simply declare $[x_1,\dots,x_d]_{n_1,\dots,n_d}$ to be 0 when some of the $x_i$ are $\infty$. Instead we take the equality in Theorem~\ref{thm:InvRel} as a \emph{definition}.
\begin{definition}\label{def:InfinityTerms}
If some of the $x_i$ are $\infty$, we \emph{define}
\begin{equation}[x_1,\dots,x_d]_{n_1,\dots,n_d}=(-1)^d(-1)^{n_1+\dots+n_d}\inv([x_d^{-1},\dots,x_1^{-1}]_{n_1,\dots,n_d})
\end{equation}
when $n_1+\cdots+n_d>1$. 
\end{definition}
For example we have
\begin{equation}
[x_1,\infty]_{n_1,n_2}=(-1)^{n_2}\binom{n_1+n_2}{n_1-1}[x_1]_{n_1+n_2},\quad  [\infty,x_2]_{n_1,n_2}=-(-1)^{n_1}\binom{n_1+n_2}{n_2-1}[x_2]_{n_1+n_2}
\end{equation}
\begin{remark}
In weight 1, symbols $[\infty]_1$ and $[\infty]_0$ are not defined.
\end{remark}

\begin{proof}[Proof of Theorem~\ref{thm:InvRel}]
The structure of the proof is similar to that of the depth 1 inversion relation in Example~\ref{ex:Depth1Inv}. Suppose by induction that it holds in depth less than $d$ and in weight less than $n$. The case of depth 1 is the relation $[x]_n+(-1)^n[x^{-1}]_n=0$. For $X=[\mathbf x_{1,\dots,d}|\mathbf t_{1,\dots,d}]$ let 
\begin{equation}\label{eq:inversion}
    X^{-1}=[\mathbf x^{-1}_{d,\dots,1}|-\mathbf t_{d,\dots,1}].
\end{equation}
We must prove that $X^{-1}=\inv(X)$. We first prove by induction that $\delta(X^{-1})_n=\delta(\inv X)_n$, which implies that $X^{-1}-\inv(X)$ is constant. Setting $x_1,\dots,x_d$ equal to zero then gives the result by Definition~\ref{def:InfinityTerms}. In order to prove that $\delta(X^{-1})_n=\delta(\inv(X))_n$, we consider all possible terms involved. It will be convenient to think of the $x_i$ appearing in each term as formal variables, and not elements in $F$. It thus makes sense to talk about \emph{inverted terms} (those that involve inverses $x_i^{-1}$) and \emph{regular terms} (those that don't involve inverses). Any identity in the free abelian group on terms (or their wedge products) gives rise to an identity in $\mathbb L^{\symb}(F)$ (or $\wedge^2(\mathbb L^{\symb}(F))$) by assigning values to the $x_i$. We shall consider three distinct operations on terms which we all think of as inversions:
\begin{equation}
    \inv,\qquad X\mapsto X^{-1},\qquad \INV,
\end{equation}
where $\inv$ is defined by~\eqref{eq:InvRel}, $X\mapsto X^{-1}$ by~\eqref{eq:inversion}, and $\INV$ by fixing regular terms and replacing an inverted term $X^{-1}$ by $\inv(X)$. Letting
\begin{equation}\label{eq:ABCDdef}
\begin{aligned}
   &A=[\mathbf x_{2,\dots,d}|\mathbf t_{2,\dots,d}]& B=[\mathbf x_{2,\dots,d}|\mathbf t_{2,\dots,d}-t_1],\\
   &C=[\mathbf x_{1,\dots,d-1}|\mathbf t_{1,\dots,d-1}], &D=[\mathbf x_{1,\dots,d-1}|\mathbf t_{1,\dots,d-1}-t_d].
   \end{aligned}
\end{equation}
we claim that the following holds in $\mathbb L^{\symb}(F)$:
\begin{equation}\label{eq:Claim}
    \INV\big(\delta(X^{-1})+(-1)^d\delta(X)\big)=\INV\delta\Big(\frac{(-1)^d}{t_1}A-\frac{(-1)^d}{t_1}B-\frac{1}{t_d}C^{-1}+\frac{1}{t_d}D^{-1}\Big).
\end{equation}
Since the cobracket of $X$ consists entirely of terms of lower depth and weight, where $Y^{-1}=\inv(Y)$ by induction, it follows that one has 
\begin{multline}
    \delta(X^{-1})_n=\delta\Big(-(-1)^dX+\frac{(-1)^d}{t_1}A-\frac{(-1)^d}{t_1}B-\frac{1}{t_d}C^{-1}+\frac{1}{t_d}D^{-1}\Big)_n=\\\delta\Big(-(-1)^dX+\frac{(-1)^d}{t_1}A-\frac{(-1)^d}{t_1}B-\frac{1}{t_d}\inv(C)+\frac{1}{t_d}\inv(D)\Big)_n=\delta\inv(X)_n
\end{multline}
in $\mathbb L_n(F)$. It remains to prove~\eqref{eq:Claim}.
Using the shorthands 
\begin{equation}
\begin{aligned}
Y_{r,s}=[\mathbf x_{r,\dots,s}|\mathbf t_{r,\dots,s}]&,\quad Y_{r,s;u}=[\mathbf x_{r,\dots,s}|\mathbf t_{r,\dots,s}-t_u] \\
Y_{r,s}^{p\to q}=[\mathbf x_{r,\dots,p\to q,\dots,s}|\mathbf t_{r,\dots,p,q+1,\dots,s}]&,\quad Y_{r,s;u}^{p\to q}=[\mathbf x_{r,\dots,p\to q,\dots,s}|\mathbf t_{r,\dots,p,q+1,\dots,s}-t_u] \\
Z_{r,s}^{p\to q}=[\mathbf x_{r,\dots,p\to q,\dots,s}|\mathbf t_{r,\dots,p-1,q,\dots,s}]&,\quad Z_{r,s;u}^{p\to q}=[\mathbf x_{r,\dots,p\to q,\dots,s}|\mathbf t_{r,\dots,p-1,q,\dots,s}-t_u]
\end{aligned}
\end{equation}
we have
\begin{equation}\label{eq:DeltaXInvPlusX}
\begin{aligned}
    \delta(X^{-1})+(-1)^d\delta(X)=& \left(X+(-1)^dX^{-1}\right)\wedge\sum_{p=1}^dt_p[x_p]_0+\\
    &\sum_{2\leq p\leq d}(-1)^p\left(Y_{p,d}+(-1)^{d-p+1}{Y^{-1}_{p,d}}\right)\wedge{Y^{-1}_{1,p-1}}+\\&\sum_{2\leq p\leq d}Y_{p,d}\wedge\left(Y_{1,p-1}+(-1)^{p-1}{Y^{-1}_{1,p-1}}\right)+\\&\sum_{1\leq p<q\leq d}\left(Y_{1,d}^{p\to q}+(-1)^{d-q+p}{(Y_{1,d}^{p\to q})}^{-1}\right)\wedge Y_{p+1,q;p}+\\&\sum_{1\leq p<q\leq d}(-1)^{q-p}\left(Z_{1,d}^{p\to q}+(-1)^{d-q+p}{(Z_{1,d}^{p\to q})}^{-1}\right)\wedge{Y^{-1}_{p,q-1;q}}.
    \end{aligned}
\end{equation}
For any regular term $W$ we can define $A_W$, $B_W$, $C_W$, and $D_W$ as in~\eqref{eq:ABCDdef}. If $W$ has depth $d_W$ the definitions of $\INV$ and $\inv$ imply that
\begin{equation}
    \INV\big(W^{-1}+(-1)^{d_W}W\big)=\INV\Big(\frac{(-1)^{d_W}}{t_1}A_W-\frac{(-1)^{d_W}}{t_1}B_W-\frac{1}{t_{d_W}}C_W^{-1}+\frac{1}{t_{d_W}}D_W\Big).
\end{equation}
Plugging this into~\ref{eq:DeltaXInvPlusX} it is now straightforward to match up the terms with those on the righthand side of~\eqref{eq:Claim}. This concludes the proof.
\end{proof}

\begin{remark} Even if we don't allow any symbols with $x_i=\infty$ we still expect Theorem~\ref{thm:InvRel} to hold, but only up to torsion. For example, one can show that $6([x]_2+[x^{-1}]_2)=0\in\mathbb L_2(F)$.
\end{remark}

\subsection{A coalgebra without inverted terms}\label{sec:AltCobracket}
As an aside, we now show that if we change $\delta$ by replacing each term $[x_{q-1}^{-1},\dots,x_p^{-1}|t_q- \mathbf t_{p,\dots,q-1}]$ by $\inv[\mathbf x_{p,\dots,q-1}|\mathbf t_{p,\dots,q-1}-t_q]$, we obtain another cobracket $\delta'$ on $\mathbb L^{\symb}(F)$. Moreover, the groups $R_n(F)$ defined using $\delta$ or $\delta'$ are the same. All that remains is to show that $\delta'^2=0$. As in the proof of Theorem~\ref{thm:InvRel} the argument is purely symbolic, and makes use of the operation $\INV$. 
\begin{lemma}\label{eq:INVdelta} For a regular term $X$ we have
\begin{equation}
    \INV\delta(X)=\delta'(X),\qquad \INV\delta(X^{-1})=\delta'(\inv(X)).
\end{equation}
\end{lemma}
\begin{proof}
The first equality follows immediately from the definition of $\delta'$. The second holds in depth 1, and the general case follows by induction from~\eqref{eq:Claim}.
\end{proof}

\begin{theorem}
$\delta'$ is a cobracket on $\mathbb L^{\symb}(F)$, i.e. $\delta'^2=0$.
\end{theorem}

\begin{proof}
Let's write $\delta(X)$ in the form $\sum_{i}X^{(i)}_1\wedge X^{(i)}_2+\sum_{j}X^{(j)}_3\wedge{X^{(j)}_4}^{-1}$. By Lemma~\ref{eq:INVdelta} we have
\begin{equation}
\begin{aligned}
0&=\INV\delta^2(X)=\INV\delta\left(\sum_{i}X^{(i)}_1\wedge X^{(i)}_2+\sum_{j}X^{(j)}_3\wedge{X^{(j)}_4}^{-1}\right)\\
&=\sum_{i}\INV\delta X^{(i)}_1\wedge X^{(i)}_2-X^{(i)}_1\wedge\INV\delta X^{(i)}_2+\sum_{j}\INV\delta X^{(j)}_3\wedge\inv X^{(j)}_4-X^{(j)}_3\wedge\INV\delta {X^{(j)}_4}^{-1}\\
&=\sum_{i}\delta' X^{(i)}_1\wedge X^{(i)}_2-X^{(i)}_1\wedge\delta' X^{(i)}_2+\sum_{j}\delta' X^{(j)}_3\wedge\inv X^{(j)}_4-X^{(j)}_3\wedge\delta'\inv X^{(j)}_4\\
&=\delta'\left(\sum_{i}X^{(i)}_1\wedge X^{(i)}_2+\sum_{j}X^{(j)}_3\wedge\inv X^{(j)}_4\right)=\delta'^2(X).
\end{aligned}
\end{equation}
This proves the result.
\end{proof}

\section{Shuffle relations}\label{sec:ShuffleRelations}
Since the multiple polylogarithms are iterated integrals they satisfy shuffle relations (see e.g.~\cite[Sec.~2.5]{Goncharov_MultiplePolylogarithmsAndMixedTateMotives}). For example, one has
\begin{equation}
    \Li_{n_1}(x_1)\Li_{n_2}(x_2)=\Li_{n_1,n_2}(x_1,x_2)+\Li_{n_2,n_1}(x_2,x_1)+\Li_{n_1+n_2}(x_1x_2).
\end{equation}
This motivates us to define \emph{shuffle products}
\begin{equation}
    [x_1]_{n_1}[x_2]_{n_2}=[x_1,x_2]_{n_1,n_2}+[x_2,x_1]_{n_2,n_1}+[x_1x_2]_{n_1+n_2} \in \mathbb L_{n_1+n_2}(F).
\end{equation}
In terms of power series this is equivalent to
\begin{equation}
    [x_1|t_1][x_2|t_2]=[x_1,x_2|t_1,t_2]+[x_2,x_1|t_2|t_1]+\frac{1}{t_1-t_2}([x_1x_2|t_1]-[x_1x_2|t_2]).
\end{equation}
Similarly, we define a shuffle product
\begin{multline}
    [x_1,x_2|t_1,t_2][x_3|t_3]=[x_1,x_2,x_3|t_1,t_2,t_3]+[x_1,x_3,x_2|t_1,t_3,t_2]+[x_3,x_1,x_2|t_3,t_1,t_2]\\
+\frac{1}{t_1-t_3}\big([x_1x_3,x_2|t_1,t_2]-[x_1x_3,x_2|t_3,t_2]\big)+\frac{1}{t_2-t_3}\big([x_1,x_2x_3|t_1,t_2]-[x_1,x_2x_3|t_1,t_3]\big)
\end{multline}
corresponding to the shuffle relations for~$\Li_{n_1,n_2}(x_1,x_2)\Li_{n_3}(x_3)$.
\begin{theorem}\label{thm:LowDepthShuffle} The shuffle products $[x_1|t_1][x_2|t_2]$ and $[x_1,x_2|t_1,t_2][x_3,t_3]$ are zero in $\mathbb L(F)$.
\end{theorem}
\begin{proof}
A simple computation shows that
\begin{equation}\label{eq:DeltaShuff11}
    \delta([x_1|t_1][x_2|t_2])=([x_1|t_1][x_2|t_2])\wedge(t_1[x_1]_0+t_2[x_2]_0)\in\wedge^2(\mathbb L(F)).
\end{equation}
Assuming by induction that $([x_1|t_1][x_2|t_2])_{n-1}=0$ it follows that $([x_1|t_1][x_2|t_2])_n$ is constant. Setting $x_1=x_2=0$ shows that the constant is 0. Another straightforward computation shows that 
\begin{multline}\label{eq:DeltaShuff21}
\delta([x_1,x_2|t_1,t_2][x_3|t_3])=([x_1,x_2|t_1,t_2][x_3|t_3])\wedge \sum_{i=1}^3t_i[x_i]_0-[x_1|t_1]\wedge([x_2|t_2][x_3|t_3])\\
+[x_2|t_2]\wedge([x_1|t_1][x_3|t_3])+([x_1x_2|t_1][x_3|t_3])\wedge[x_2|t_2-t_1]-([x_1x_2|t_2][x_3|t_3])\wedge([x_1|t_1-t_2]+[x_1]_0)\\
+[x_1x_2x_3|t_1]\wedge([x_2|t_2-t_1] [x_3|t_3-t_1])-[x_1x_2x_3|t_2]\wedge([x_1|t_1-t_2][x_3|t_3-t_2]),
\end{multline}
so the same argument as above shows that $[x_1,x_2|t_1,t_2][x_3,t_3]=0$ as well.
\end{proof}

\begin{remark} We can similarly define shuffle products in arbitrary depth. We conjecture that they are all zero in $\mathbb L(F)$. Proving this would require showing that the cobracket of a shuffle relation is given in terms of shuffle relations in lower weight and depth. This appears to be the case.
\end{remark}

\begin{remark}\label{rm:ShuffleZeroInLsymb}
It seems worth noting that if one replaces $\delta$ by $\delta'$, \eqref{eq:DeltaShuff11} and \eqref{eq:DeltaShuff21} hold not just in $\wedge^2(\mathbb L(F))$, but also $\wedge^2(\mathbb L^{\symb}(F))$ (and not just modulo 2-torsion). 
\end{remark}

\section{The Goncharov--Rudenko Lie coalgebra}\label{sec:GoncharovRudenko}
Goncharov and Rudenko~\cite{GoncharovRudenko} also considered the problem of explicitly constructing the motivic Lie-coalgebra. They give an explicit construction in weight 4 and less. We denote their coalgebra $\mathbb L^{GR}_{\leq 4}(F)$. They define $\mathbb L^{GR}_1(F)=F^*$, $\mathbb L^{GR}_2(F)=\B_2(F)$, and define (for $n=3,4$) $\mathbb L^{GR}_n(F)$ to be the group generated by symbols $\{x\}_n$ and $\{x,y\}_{n-1,1}$ subject to certain relations, most notably a relation $Q_n$. The relationship with their work and ours is given in Theorem~\ref{thm:RelationToGR} below. Its proof is straightforward, but long, so we omit it. The main point is to show that $Q_n$ maps to an element in $R_n(F)$. We stress that their symbol $\{x,y\}_{n-1,1}$ does not correspond to our symbol $[x,y]_{n-1,1}$. We find our symbol more natural since it directly relates to the multiple polylogarithm (their symbol is related to motivic correlators).
\begin{theorem}\label{thm:RelationToGR} There is a surjective map $\mathbb L^{GR}_{\leq 4}(F)_\Q\to \mathbb L_{\leq 4}(F)_\Q$ taking $\{x\}_n$ to $[x]_n$, $\{x,y\}_{2,1}$ to $-[x/y,y]_{2,1}-[x]_3-[y]_3$ and $\{x,y\}_{3,1}$ to $-[x/y,y]_{3,1}-[x]_4+[y]_4$. The map preserves the cobracket.
\end{theorem}
\begin{remark} We suspect that the map is an isomorphism, but we do not know if all of our relations can be expressed in terms of the $Q_n$ relations.
\end{remark}

\bibliographystyle{alpha}
\bibliography{Bibliography}

\begin{thebibliography}{Gon05}

\bibitem[DB12]{Deligne_InterpretationMotiviqueDeLaConjectureDeZagierReliantPolylogarithmesEtRegulateurs}
Pierre Deligne and Alexander~A. Beilinson.
\newblock Interpr\'etation motivique de la conjecture de {Z}agier reliant
  polylogarithmes et r\'egulateurs.
\newblock {\em J. High Energy Phys.}, 2012.

\bibitem[Gon94]{GoncharovMotivicGalois}
A.~B. Goncharov.
\newblock Polylogarithms and motivic {G}alois groups.
\newblock In {\em Motives ({S}eattle, {WA}, 1991)}, volume~55 of {\em Proc.
  Sympos. Pure Math.}, pages 43--96. Amer. Math. Soc., Providence, RI, 1994.

\bibitem[Gon95]{GoncharovConfigurations}
A.~B. Goncharov.
\newblock Geometry of configurations, polylogarithms, and motivic cohomology.
\newblock {\em Adv. Math.}, 114(2):197--318, 1995.

\bibitem[Gon01]{Goncharov_MultiplePolylogarithmsAndMixedTateMotives}
Alexander Goncharov.
\newblock Multiple polylogarithms and mixed {T}ate motives.
\newblock {\em arXiv.math/0103059}, 2001.

\bibitem[Gon05]{Goncharov_GaloisSymmetriesOfFundamentalGroupoidsAndNoncommutativeGeometry}
Alexander Goncharov.
\newblock Galois symmetries of fundamental groupoids and noncommutative
  geometry.
\newblock {\em Duke Math. J.}, 128(2):209--284, 2005.

\bibitem[GR18]{GoncharovRudenko}
Alexander~B. Goncharov and Daniil Rudenko.
\newblock Motivic correlators, cluster varieties and {Z}agier's conjecture on
  zeta({F},4).
\newblock {\em arXiv:1803.08585}, 2018.

\bibitem[Lew91]{PolylogBook}
Leonard Lewin, editor.
\newblock {\em Structural properties of polylogarithms}, volume~37 of {\em
  Mathematical Surveys and Monographs}.
\newblock American Mathematical Society, Providence, RI, 1991.

\bibitem[Zag91]{ZagierDedekindZeta}
Don Zagier.
\newblock Polylogarithms, {D}edekind zeta functions and the algebraic
  {$K$}-theory of fields.
\newblock In {\em Arithmetic algebraic geometry ({T}exel, 1989)}, volume~89 of
  {\em Progr. Math.}, pages 391--430. Birkh\"{a}user Boston, Boston, MA, 1991.

\bibitem[Zha02]{ZhaoMultiplePolylogsMonodromyHodgeStructures}
Jianqiang Zhao.
\newblock Multiple polylogarithms: analytic continuation, monodromy, and
  variations of mixed {H}odge structures.
\newblock In {\em Contemporary trends in algebraic geometry and algebraic
  topology ({T}ianjin, 2000)}, volume~5 of {\em Nankai Tracts Math.}, pages
  167--193. World Sci. Publ., River Edge, NJ, 2002.

\bibitem[Zic19]{Zickert_HolomorphicPolylogarithmsAndBlochComplexes}
Christian~K. Zickert.
\newblock Holomorphic polylogarithms and {B}loch complexes.
\newblock {\em arXiv:1902.03971}, 2019.

\end{thebibliography}

\end{document}